\documentclass[1p,number,sort&compress]{elsarticle}%%%,draft
\usepackage{amsmath, mathrsfs, amssymb, enumerate, txfonts, color}
\numberwithin{equation}{section}
\allowdisplaybreaks
\newtheorem{theorem}{Theorem}[section]
\newtheorem{lemma}[theorem]{Lemma}
\newtheorem{cor}[theorem]{Corollary}

\newdefinition{definition}[theorem]{Definition}
\newdefinition{remark}[theorem]{Remark}
\newdefinition{example}[theorem]{Example}
\newproof{proof}{Proof}
\newcommand{\wconv}{\rightharpoonup}
\newcommand{\wsconv}{\mathrel{\vbox{\offinterlineskip\ialign{\hfil##\hfil\cr $\hspace{-0.4ex}\textnormal{\scriptsize{*}}$\cr \noalign{\kern-0.6ex}  $\rightharpoonup$\cr}}}}
\newcommand*\diff{\mathop{}\!\mathrm{d}}

\begin{document}
\begin{frontmatter}
\title{On a multivalued differential equation\\ with nonlocality in time\footnote{This work has been supported by Deutsche Forschungsgemeinschaft through Collaborative Research Center 910
``Control of self-organizing nonlinear systems: Theoretical methods and concepts of application''.
}}
\author{Andr\'e Eikmeier}
\ead{eikmeier@math.tu-berlin.de}
\author{Etienne Emmrich\corref{ee}}
\cortext[ee]{Corresponding author}
\ead{emmrich@math.tu-berlin.de}
\address{Technische Universit{\"a}t Berlin, Institut f\"ur Mathematik, Stra{\ss}e des 17.~Juni 136, 10623 Berlin, Germany}%\\
%%%%%%%%%%%%%%%%%%%%%%%%%%%%%%%%%%%%%%%%%%%%%%%%%%%%%%%%
%%\received{(Day Month Year)}
%%\revised{(Day Month Year)}
%%\accepted{(Day Month Year)}
%%%%%%%%%%%%%%%%%%%%%%%%%%%%%%%%%%%%%%%%%%%%%%%%%%%%%%%%
%%%%%%%%%%%%%%%%%%%%%%%%%%%%%%%%%%%%%%%%%%%%%%%%%%%%%%%%

\begin{abstract}
The initial value problem for a multivalued differential equation is studied, which is governed by the sum of a monotone, hemicontinuous, coercive operator fulfilling a certain growth condition and a Volterra integral operator in time of convolution type with exponential decay. The two operators act on different Banach spaces where one is not embedded in the other. The set-valued right-hand side is measurable and satisfies certain continuity and growth conditions. Existence of a solution is shown via a generalisation of the Kakutani fixed-point theorem.
\end{abstract}
%%%%%%%%%%%%%%%%%%%%%%%%%%%%%%%%%%%%%%%%%%%%%%%%%%%%%%%%
%%%%%%%%%%%%%%%%%%%%%%%%%%%%%%%%%%%%%%%%%%%%%%%%%%%%%%%%
\end{frontmatter}
%%%%%%%%%%%%%%%%%%%%%%%%%%%%%%%%%%%%%%%%%%%%%%%%%%%%%%%%
%%%%%%%%%%%%%%%%%%%%%%%%%%%%%%%%%%%%%%%%%%%%%%%%%%%%%%%%
\section{Introduction} \label{introduction}
\subsection{Problem statement and main result}

\noindent We consider the multivalued differential equation\footnote{Also named differential inclusion by many authors. However, in order to distinguish this kind of problem from the ones containing subdifferentials or set-valued (maximal monotone) operators, we chose the name multivalued differential equation.}
\begin{equation} \label{problem}
  \begin{split}
    v'(t) + Av(t) +(BKv)(t) &\in F(t,v(t)), \quad t\in(0,T),\\
    v(0)&=v_0,
  \end{split}
\end{equation}
where
\begin{equation} \label{defK}
  (Kv)(t) = u_0+\int_0^t k(t-s) v(s) \diff s, \quad k(z)=\lambda e^{-\lambda z}.
\end{equation}
%with an arbitrary $k\in L^\infty(0,T)$. 
Here, $T>0$ defines the considered time interval, $\lambda>0$ is a given parameter and $v_0$, $u_0$ are the given initial data of the problem.

The operator $A\colon V_A\to V_A^*$ is a monotone, hemicontinuous, coercive operator satisfying a certain growth condition, where $V_A$ is a real, reflexive Banach space. The operator $B\colon V_B\to V_B^*$ is linear, bounded, strongly positive, and symmetric, where $V_B$ denotes a real Hilbert space.
The space $V_A$ shall be compactly and densely embedded in a real Hilbert space $H$, whereas $V_B$ shall be only continuously and densely embedded in $H$. The dual of $H$ is identified with $H$ itself, such that both $V_A$, $H$, $V_A^*$ and $V_B$, $H$, $V_B^*$ form a so-called Gelfand triple. However, we do not assume any relation between $V_A$ and $V_B$ apart from $V=V_A\cap V_B$ being separable and densely embedded in both $V_A$ and $V_B$. We do not assume that $V_A$ is embedded into $V_B$ or the other way around. Overall, we have the scale
\begin{equation}\label{embed}
V_A\cap V_B=V\subset V_C \subset H = H' \subset V_C'\subset V'= V_A'+V_B',\quad C\in \{A,B\},
\end{equation}
of Banach and Hilbert spaces, where all embeddings are meant to be continuous and dense and the embedding $V_A\subset H$ is even meant to be compact.

The operator $F\colon [0,T]\times H\to P_{fc}(H)$ is graph measurable, fulfils a certain growth condition in the second argument and the graph of $v\mapsto F(t,v)$ is sequentially closed in $H\times H_w$ for almost all $t\in(0,T)$, where $H_w$ denotes the Hilbert space $H$ equipped with the weak topology. The set $P_{fc}(H)$ denotes the set of all nonempty, closed, and convex subsets of $H$.
 
Multivalued differential equations appear, e.g., in the formulation of optimal feedback control problems. If we consider the inclusion as an equation with a side condition on the right-hand side, i.e.,
\begin{equation*}
  \begin{split}
    v'(t) + Av(t) +(BKv)(t) &=f(t), \phantom{F(t,v(t))} \hspace{-0.3cm} t\in(0,T), \\
    f(t)&\in F(t,v(t)), \phantom{f(t)}\hspace{-0.3cm} t\in(0,T),\\
    v(0)&=v_0,
  \end{split}
\end{equation*}
we can consider $f$ as the control of our system with the corresponding state $v$ and $F$ as the set of admissible controls, which, in the case of $F$ depending on $v$, leads to a feedback control system.

Physical applications of the system we are considering in this work are, e.g., heat flow in materials with memory (see, e.g., MacCamy~\cite{MacCamy77a}, Miller~\cite{Miller78}) or viscoelastic fluid flow (see, e.g., Desch, Grimmer, and Schappacher~\cite{Desch88}, MacCamy~\cite{MacCamy77b}). Another application related to that are non-Fickian diffusion models which describe diffusion processes of a penetrant through a viscoelastic material (see, e.g., Edwards~\cite{Edwards96}, Edwards and Cohen~\cite{EdwardsCohen95}, Shaw and Whiteman~\cite{ShawWhiteman98}). They also appear, e.g., in mathematical biology (see, e.g., Cushing~\cite{Cushing}, Fedotov and Iomin~\cite{FedotovIomin08}, Mehrabian and Abousleiman~\cite{MehrabianAbousleiman11}).

Due to the specific form of the kernel $k$ given in \eqref{defK}, we can rewrite our system into the coupled system
\begin{equation} \label{coupled}
  \begin{split}
    v'(t) + Av(t) + Cu(t) &\in F(t,v(t)) , \phantom{=\lambda Dv(t)}\hspace{-1.0cm} t\in(0,T),\\
    (u-Du_0)'(t) + \lambda (u-Du_0)(t)  &=\lambda Dv(t), \phantom{\in F(t,v(t))}\hspace{-1.0cm} t\in(0,T),\\
    v(0)&=v_0, \\
    u(0)&=Du_0,
  \end{split}
\end{equation}
where $C$ and $D$ are suitably chosen linear operators such that $B=CD$.

Instead of the kernel $k(z)=\lambda e^{-\lambda z}$, we might also consider $k(z)=c e^{-\lambda z}$ with $c,\lambda>0$. However, for simplicity, we will stick to the first type of kernel. Actually, this type appears naturally in many applications. In these applications, $\frac{1}{\lambda}$ is often describing a relaxation or averaged delay time. If we consider the limit $\lambda\to 0$, the system \eqref{coupled} decouples such that $u(t)=Du_0$, $t\in [0,T]$, is the solution of the second equation. In the case $\lambda\to \infty$, the system reduces to a single first-order equation for $v$ without memory. 

In the case of the kernel $k(z)=c e^{-\lambda z}$, the behaviour for $\lambda\to 0$ is slightly different. The limit then yields a second-order in time equation for $u$ (see, e.g., Emmrich and Thalhammer~\cite{EmmrichThalhammer1}).
 
%%%%%%%%%%%%%%%%%%%%%%%%%%%%%%%%%%%%%%%%%%%%%%%%%%%%%%%%
\subsection{Literature overview}

This work is a continuation of Eikmeier, Emmrich, and Kreusler~\cite{EikEmmKre}. There, the single-valued instead of the multivalued differential equation is considered in the same setting concerning the spaces $V_A$ and $V_B$. However, due to the structure of the proof in the present work, we additionally need the compact embedding $V_A\subset H$ and we have to assume that the right-hand side is pointwisely $H$-valued.

Nonlinear integro-differential equations have been considered by many authors through the years. Results on well-posedness for more general classes of nonlinear evolution equations including Volterra operators, but only in the case of Hilbert spaces $V_A=V_B$, can be found in, e.g., Gajewski, Gr\"oger, and Zacharias~\cite{GGZ}. In contrast to this, Crandall, Londen, and Nohel~\cite{CrandallLondenNohel} study the case of a doubly nonlinear problem, where both nonlinear operators are assumed to be (possibly multivalued) maximal monotone subdifferential operators and the domain of definition of one of them has to be continuously and densely embedded in the domain of definition of the other one. For more references on nonlinear and also linear evolutionary integro-differential equations see Eikmeier, Emmrich, and Kreusler~\cite[Section 1.2]{EikEmmKre}.

Multivalued differential equations have also been studied by various authors. Basic results, also for set-valued analysis, can be found in, e.g., Aubin and Cellina~\cite{AubinCellina}, Aubin and Frankowska~\cite{AubinFrankowska}, or Deimling~\cite{Deimling}. In O'Regan~\cite{ORegan}, some extensions of the results shown in Deimling~\cite{Deimling} are presented. A semilinear multivalued differential equation with a linear, bounded, and strongly positive operator and a set-valued nonlinear operator is, e.g., considered in Beyn, Emmrich, and Rieger~\cite{BeynEmmRie}.

In particular, integro-differential equations in the multivalued case have been studied by, e.g., Papageorgiou~\cite{Papageorgiou88b,Papageorgiou88a,Papageorgiou91,Papageorgiou91b}. The equations are considered under different assumptions with the set-valued operator appearing in the integral term. In most of the works mentioned, examples of applications in the theory of optimal control are given. 

%\begin{equation*}
%  \begin{split}
%    v'(t)+A(t,v(t))&\in \int_0^t k(t-s) F(s,v(s)) \diff s, \quad t\in (0,T), \\
%    v(0)&=v_0,
%  \end{split}
%\end{equation*}
%with an arbitrary $k\in L^\infty(0,T;L(H))$, where $L(H)$ denotes the space of linear and bounded operators mapping $H$ into $H$, is considered and existence of solutions is proven via a generalisation of the Kakutani fixed-point theorem.

The optimal feedback control of a motion of a viscoelastic fluid via a multivalued differential equation is, e.g., considered in Gori et al.~\cite{Gori_etal} and Obukhovski\u{\i}, Zecca, and Zvyagin~\cite{ObZecZvy}. Existence of solutions for the equation are shown via topological degree theory.

%%%%%%%%%%%%%%%%%%%%%%%%%%%%%%%%%%%%%%%%%%%%%%%%%%%%%%%%
%%%%%%%%%%%%%%%%%%%%%%%%%%%%%%%%%%%%%%%%%%%%%%%%%%%%%%%%
\subsection{Organisation of the paper}

\noindent
The paper is organised as follows: In Section~\ref{sec_notation}, we introduce the general notation and some basic results from set-valued analysis. In Section~\ref{Assumptions}, we state our assumptions on the operators $A$, $B$, and $F$ and some preliminary results concerning properties we need in the following Section~\ref{existence}, where we prove existence of a solution to problem \ref{problem}. This is done via a generalisation of the Kakutani fixed-point theorem.

%%%%%%%%%%%%%%%%%%%%%%%%%%%%%%%%%%%%%%%%%%%%%%%%%%%%%%%%
%%%%%%%%%%%%%%%%%%%%%%%%%%%%%%%%%%%%%%%%%%%%%%%%%%%%%%%%
\section{Notation} \label{sec_notation}

\noindent Let $X$ be a Banach space with its dual $X^*$. The norm in $X$ and the standard norm in $X^*$ are denoted by $\Vert \cdot\Vert_X$ and $\Vert \cdot\Vert_{X^*}$, respectively. The duality pairing between $X$ and $X^*$ is denoted by $\langle\cdot,\cdot\rangle$. If $X$ is a Hilbert space, the inner product in $X$ is denoted by $(\cdot,\cdot)$. For the intersection $X\cap Y$ of two Banach spaces $X$ and $Y$, we consider the norm $\Vert \cdot\Vert_{X\cap Y}= \Vert \cdot\Vert_X+\Vert \cdot\Vert_Y$, and for the sum $X+Y$, we consider the norm 
\begin{equation*}
  \Vert z\Vert_{X+Y} = \inf \left\{ \max \left( \Vert z_X\Vert_X,\ \Vert z_Y\Vert_Y\right) \mid z=z_X+z_Y\ \text{with}\ z_X\in X,\ z_Y\in Y\right\}\!.
\end{equation*}
Note that $(X\cap Y)^*=X^*+Y^*$ if $X$ and $Y$ are embedded in a locally convex space and $X\cap Y$ is dense in $X$ and $Y$ with respect to the norm above, see, e.g., Gajewski~et~al.~\cite[pp.~12ff.]{GGZ}.

Now, let $X$ be a real, reflexive, and separable Banach space and $1\leq p\leq \infty$. By $L^p(0,T;X)$, we denote the usual space of Bochner measurable (sometimes also called strongly measurable), $p$-integrable functions equipped with the standard norm. For $1\leq p<\infty$, the duality pairing between $L^p(0,T;X)$ and its dual space $L^q(0,T;X^*)$, where $\frac{1}{p}+\frac{1}{q}=1$ for $p>1$ and $q=\infty$ for $p=1$, is also denoted by $\langle \cdot,\cdot\rangle$, and it is given by
\begin{equation*}
  \langle g,f\rangle = \int_0^T \langle g(t),f(t)\rangle\diff t,
\end{equation*}
see, e.g., Diestel and Uhl~\cite[Theorem 1 on p. 98, Corollary 13 on p. 76, Theorem 1 on p. 79]{DiestelUhl}. 

By $W^{1,p}(0,T;X)$, $1\leq p\leq \infty$, we denote the usual space of weakly differentiable functions $u\in L^p(0,T;X)$ with $u'\in L^p(0,T;X)$, equipped with the standard norm. By $\mathscr{C}([0,T];X)$, we denote the space of functions that are continuous on $[0,T]$ with values in $X$, whereas $\mathscr{C}_w([0,T];X)$ denotes the space of functions that are continuous on $[0,T]$ with respect to the weak topology in $X$. We have the continuous embedding $W^{1,1}(0,T;X) \subset \mathscr{C}([0,T];X)$, see, e.g., Roub\'i\v{c}ek~\cite[Lemma 7.1]{Roubicek}. Furthermore, a function $u\in W^{1,1}(0,T;X)$ is almost everywhere equal to a function that is absolutely continuous on $[0,T]$ with values in $X$, see, e.g., Br\'ezis~\cite[Theorem 8.2]{Brezis}. We denote the set of all these functions by $\mathscr{AC}([0,T];X)$. By $\mathscr{C}^1([0,T])$, we denote the space of on $[0,T]$ continuously differentiable real-valued functions. By $c$, we denote a generic positive constant.

Now, let us recall some definitions from set-valued analysis. Let $(\Omega, \Sigma)$ be a measurable space and let $X$ be a complete separable metric space. By $\mathcal{L}([a,b])$ and $\mathcal{B}(X)$, we denote the Lebesgue $\sigma$-algebra on the interval $[a,b]\subset \mathbb{R}$ and the Borel $\sigma$-algebra on $X$, respectively. By $P_{f}(X)$, we denote the set of all nonempty and closed subsets $U\subset X$, and by $P_{fc}(X)$, we denote the set of all nonempty, closed, and convex subsets $U\subset X$.

For a set-valued function $F\colon \Omega\to 2^X\setminus \{\emptyset\}$, let
\begin{equation*}
  \vert F(\omega)\vert:=\sup \left\{\Vert x\Vert_X\mid x\in F(\omega)\right\},\quad \omega\in \Omega.
\end{equation*}
A function $F\colon \Omega\to P_f(X)$ is called measurable (sometimes also called weakly measurable) if the preimage of each open set is measurable, i.e.,
\begin{equation*}
  F^{-1}(U):=\left\lbrace \omega\in \Omega\mid F(\omega)\cap U \neq \emptyset\right\rbrace \in \Sigma
\end{equation*}
for every open $U\subset X$.\footnote{Depending on the assumptions on $(\Omega,\Sigma)$ and $X$, there are many equivalent definitions of measurability for set-valued functions, see, e.g., Denkowski, Mig\'orski, and Papageorgiou~\cite[Theorem 4.3.4]{DenMigPapa}.} A function $f\colon \Omega\to X$ is called measurable selection of $F$ if $f(\omega)\in F(\omega)$ for all $\omega\in \Omega$ and $f$ is measurable. Each measurable set-valued function has a measurable selection, see, e.g., Aubin and Frankowska~\cite[Theorem 8.1.3]{AubinFrankowska}.

Now, let $(\Omega, \Sigma, \mu)$ be a complete $\sigma$-finite measure space and let $X$ be a separable Banach space. For a set-valued function $F\colon \Omega \to P_{f}(X)$ and $p\in [1,\infty)$, we denote by $\mathcal{F}^p$ the set of all $p$-integrable selections of $F$, i.e.,
\begin{equation*}
  \mathcal{F}^p:=\left\{ f\in L^p(\Omega;X,\mu) \mid f(\omega)\in F(\omega)\ \text{a.e. in}\ \Omega\right\}\!,
\end{equation*}
where $L^p(\Omega;X,\mu)$ denotes the space of Bochner measurable, $p$-integrable functions with respect to $\mu$.\footnote{Note that in the case of a separable Banach space $X$, the Bochner measurability of $f$ coincides with the $\Sigma$-$\mathcal{B}(X)$-measurability, see, e.g., Amann and Escher~\cite[Chapter X, Theorem 1.4]{AmannEscher}, Denkowski, Mig\'orski, and Papageorgiou~\cite[Corollary~3.10.5]{DenMigPapa}, or Papageorgiou and Winkert~\cite[Theorem 4.2.4]{PapaWin}} If $F$ is integrably bounded, i.e., there exists a nonnegative function $m\in L^p(\Omega;\mathbb{R},\mu)$ such that $F(\omega)\subset m(\omega)\;\! B_X$ for $\mu$-almost all $\omega\in \Omega$, where $B_X$ denotes the unit ball in $X$, each measurable selection of $F$ is in $\mathcal{F}^p$ due to Lebesgue's theorem on dominated convergence. The integral of $F$ is defined as 
\begin{equation*}
  \int_\Omega F \diff \mu := \left\{ \int_\Omega f \diff \mu \mid f\in \mathcal{F}^1\right\}\!.
\end{equation*}
For properties of this integral, see, e.g., Aubin and Frankowska~\cite[Chapter 8.6]{AubinFrankowska}.

For a set-valued function $F\colon \Omega \times X \to P_f(X)$, a function $v\colon \Omega \to X$ and $p\in [1,\infty)$, we denote by $\mathcal{F}^p(v)$ the set of all $p$-integrable selections of the mapping $\omega \mapsto F(\omega, v(\omega))$, i.e.,
\begin{equation*}
  \mathcal{F}^p(v):=\left\{ f\in L^p(\Omega;X,\mu) \mid f(\omega)\in F(\omega, v(\omega))\ \text{a.e. in}\ \Omega\right\}\!.
\end{equation*}

Finally, let $X$, $Y$ be Banach spaces and $\Omega\subset Y$. A set-valued function $F\colon \Omega\to 2^X\setminus \{\emptyset\}$ is called upper semicontinuous if $F^{-1}(U)$ is closed in $\Omega$ for all closed $U\subset X$. 

\section{Main assumptions and preliminary results}\label{Assumptions}

\noindent Throughout this paper, let $V_A$ be a real, reflexive Banach space and let $V_B$ and $H$ be real Hilbert spaces, respectively. As mentioned in Section \ref{introduction}, we require that $V=V_A\cap V_B$ is separable and the embeddings stated in \eqref{embed} are fulfilled (with the embedding $V_A\subset H$ meant to be compact). Let also $2\leq p <\infty$, $1<q\leq 2$ with $\frac{1}{p}+\frac{1}{q}=1$.

For $A\colon V_A \to V_A^*$, we say the assumptions \textbf{(A)} are fulfilled if 
\begin{itemize}
  %\item $t\mapsto A(t,v)$ is measurable,
  \item[i)] $A$ is monotone,
  \item[ii)] $A$ is hemicontinuous, i.e., $\theta\mapsto \langle A(u+\theta v, w\rangle$ is continuous on $[0,1]$ for all $u,v,w\in V_A$,
  \item[iii)] $A$ fulfils a growth condition of order $p-1$, i.e., there exists $\beta_A>0$ such that \[\Vert Av\Vert_{V_A^*} \leq \beta_A \;\!(1+\Vert v\Vert_{V_A}^{p-1})\] for all $v\in V_A$,
  \item[iv)] $A$ is $p$-coercive, i.e., there exist $\mu_A>0$, $c_A\geq 0$ such that \[\langle Av,v\rangle \geq \mu_A \;\!\Vert v\Vert^p-c_A\] for all $v\in V_A$.
\end{itemize}
One operator satisfying these assumptions is, e.g., the $p$-Laplacian $-\Delta_p=-\nabla\cdot(\vert\nabla\vert^{p-2}\nabla)$ acting between the standard Sobolev spaces $W_0^{1,p}(\Omega)$ and $W^{-1,p}(\Omega)$ for a bounded Lipschitz domain $\Omega$, see, e.g., Zeidler~\cite[p.~489]{ZeidlerIIB}.
For $B\colon V_B\to V_B^*$, we say the assumptions \textbf{(B)} are fulfilled if
\begin{itemize}
  \item[i)] $B$ is linear,
  \item[ii)] $B$ is bounded, i.e., there exists $\beta_B>0$ such that \[\Vert Bv\Vert_* \leq \beta_B \;\! \Vert v\Vert\] for all $v\in V_B$,
  \item[iii)] $B$ is strongly positive, i.e., there exists $\mu_B>0$ such that \[\langle Bv,v\rangle \geq \mu_B \;\!\Vert v\Vert^2\] for all $v\in V_B$,
  \item[iv)] $B$ is symmetric.
\end{itemize}
Following these assumptions, $B$ induces a norm $\Vert\cdot\Vert_B:=\langle B\cdot, \cdot\rangle^{1/2}$ in $V_B$ that is equivalent to $\Vert \cdot\Vert_{V_B}$. Therefore, we denote the space $L^2(0,T;(V_B,\Vert\cdot\Vert_B))$ by $L^2(0,T;B)$. An example for an operator satisfying these assumptions is the Laplacian $-\Delta$ acting between the standard Sobolev spaces $H_0^{1}(\Omega)$ and $H^{-1}(\Omega)$, again for a bounded Lipschitz domain $\Omega$, as well as the fractional Laplacian $(-\Delta)^s$ for $\frac{1}{2}<s<1$, acting between the standard Sobolev--Slobodecki\u{\i} spaces $H_0^s(\Omega)$ and $H^{-s}(\Omega)$.

Finally, we say that $F\colon [0,T]\times H \to P_{fc}(H)$ fulfils the assumptions \textbf{(F)} if
\begin{itemize}
  \item[i)] $F$ is measurable,
  \item[ii)] for almost all $t\in(0,T)$, the graph of the mapping $v\mapsto F(t,v)$ is sequentially closed in $H\times H_w$, where $H_w$ denotes the Hilbert space $H$ equipped with the weak topology,
  \item[iii)] $\vert F(t,v)\vert \leq a(t) + b\;\!\Vert v\Vert_H^{2/q}$ a.e. with $a\in L^q(0,T)$, $a(t)\geq 0$ a.e. and $b>0$.
\end{itemize}
    
%Assumptions on $k$: $k\in L^\infty(0,T)$

Note that it is also possible to consider $A\colon [0,T]\times V_A\to V_A^*$ and $B\colon [0,T]\times V_B\to V_B^*$, where the mappings $t\mapsto A(t,v)$, $v\in V_A$, and $t\mapsto B(t,v)$, $v\in V_B$, are assumed to be measurable and all the assumptions above are assumed to hold uniformly in $t$. However, for simplicity, we will only consider the case of autonomous operators $A$ and $B$.

These operators can be extended to operators defined on $L^p(0,T;V_A)$ and $L^1(0,T;V_B)$, respectively. The monotonicity and hemicontinuity of $A\colon V_A\to V_A^*$ imply demicontinuity (see, e.g., Zeidler~\cite[Propos.~26.4 on
p.~555]{ZeidlerIIB}). Due to the separability of $V_A^*$, the theorem of Pettis (see, e.g., Diestel and Uhl~\cite[Thm.~2 on p.~42]{DiestelUhl}) then implies that $A$ maps Bochner measurable functions $v\colon [0,T]\to V_A$ into Bochner measurable functions $Av\colon [0,T]\to V_A^*$, where $(Av)(t)=Av(t)$ for almost all $t\in (0,T)$. Due to the growth condition, we have the estimate
\begin{equation} \label{estimateA}
  \Vert Av\Vert_{L^q(0,T;V_A^*)}\leq c (1+ \Vert v\Vert^{p-1}_{L^p(0,T;V_A)})
\end{equation}
for all $v\in L^p(0,T;V_A)$, i.e., $A$ maps $L^p(0,T;V_A)$ into $L^q(0,T;V_A^*)$.

Via the same definition $(Bv)(t)=Bv(t)$ for a function $v\colon [0,T]\to V_B$, we can extend the operator \mbox{$B\colon V_B\to V_B^*$} to a linear, bounded, strongly positive, and symmetric operator mapping $L^2(0,T;V_B)$ into its dual or to a linear, bounded operator mapping $L^r(0,T;V_B)$ into $L^r(0,T;V_B^*)$, $1\leq r \leq \infty$, respectively.

Due to the definition \eqref{defK} of the operator $K$, we have the following lemma.

\begin{lemma} \label{estimateK}
  Let $X$ be an arbitrary Banach space, $k(z)=\lambda e^{-\lambda z}$, $\lambda>0$, $u_0\in X$. The operator $K\colon L^2(0,T;X)\to L^2(0,T;X)$ is well-defined, affine-linear, and bounded. The estimate  
  \begin{equation*}
    \|Kv-u_0\|_{L^2(0,T;X)} \le \|k\|_{L^1(0,T)} \|v\|_{L^2(0,T;X)}
  \end{equation*}
  is satisfied for all $v \in L^2(0,T;X)$, where $\|k\|_{L^1(0,T)}= 1-e^{-\lambda T}$. Further, the estimate
  \begin{equation*}
    \|Kv-u_0\|_{\mathscr{C}([0,T];X)} \le \lambda\, \|v\|_{L^1(0,T;X)}
  \end{equation*}
  is satisfied for all $v \in L^1(0,T;X)$, i.e., $K$ is also an affine-linear, bounded operator of $L^1(0,T;X)$ into $\mathscr{C}([0,T];X)$ (even  $\mathscr{AC}([0,T];X)$).
\end{lemma}
The proof is based on simple calculations, therefore we omit it here. Following this lemma, we obtain the following properties of the operator $BK$.

\begin{cor}
  Let the assumptions of Lemma \ref{estimateK} (with $X=V_B$) and assumption \textbf{(B)} be fulfilled. Then the operator $BK\colon L^2(0,T;V_B)\to L^2(0,T;V_B^*)$ is well-defined, affine-linear, and bounded. The same holds for $BK\colon L^1(0,T;V_B)\to \mathscr{C}([0,T];V_B^*)$.
\end{cor}

One crucial relation in this setting, resulting from the exponential kernel, is the following one. Let $X$ be an arbitrary Banach space and $v\in L^1(0,T;X)$. Then we have
\begin{equation} \label{RelationK}
  (Kv)'(t)=\lambda \left( v(t) - \left((Kv)(t) - u_0\right) \right)
\end{equation}
for almost all $t\in (0,T)$.

Concerning the operator $F$, we need a measurability result in order to be able to extract measurable selections of the multivalued mapping $t\mapsto F(t,u(t))$, where $u$ itself is a measurable function.

\begin{lemma} \label{MeasurabilityNemytskii}
  Let $X$ be a separable Banach space, let $F\colon [0,T]\times X\to P_f(X)$ be measurable and let $v\colon [0,T]\to X$ be Bochner measurable. Then the mapping $\tilde{F}_v\colon [0,T]\to P_f(X)$, $t\mapsto F(t,v(t))$, is measurable.
\end{lemma}

\begin{proof}
  Let $U\subset X$ be open. Consider
  \begin{equation*}
    \begin{split}
      \tilde{F}^{-1}_v(U)&= \left\lbrace t\in [0,T] \mid F(t,v(t))\cap U\neq \emptyset \right\rbrace \\
      &= \pi_{[0,T]}\left( \left\lbrace (t,x)\in [0,T]\times X \mid F(t,x)\cap U \neq \emptyset , x=v(t)\right\rbrace \right) \\
      &= \pi_{[0,T]}\left( \left\lbrace (t,x)\in [0,T]\times X \mid F(t,x)\cap U \neq \emptyset \right\rbrace \cap \mathrm{graph}(v) \right)\!,
    \end{split}
  \end{equation*}
  where $\pi_{[0,T]}$ denotes the projection onto $[0,T]$. Since $v$ is Bochner measurable, its graph belongs to $\mathcal{L}([0,T])\otimes \mathcal{B}(X)$, see, e.g., Castaing and Valadier~\cite[Theorem III.36]{CastaingValadier}. Note again that for a separable Banach space $X$, Bochner measurability and $\mathcal{L}([0,T])$-$\mathcal{B}(X)$-measurability are equivalent, see, e.g., Denkowski, Mig\'orski, and Papageorgiou~\cite[Corollary~3.10.5]{DenMigPapa}. Due to the measurability of $F$, the intersection space in the equation above also belongs to $\mathcal{L}([0,T])\otimes \mathcal{B}(X)$. Since the projection maps measurable sets into measurable sets (at least in this setting, see, e.g., Castaing and Valadier~\cite[Theorem III.23]{CastaingValadier}), we have $\tilde{F}^{-1}_v(U)\in \mathcal{L}([0,T])$, which finishes the proof.\qed
\end{proof}

%%%%%%%%%%%%%%%%%%%%%%%%%%%%%%%%%%%%%%%%%%%%%%%%%%%%%%%%
\section{Existence of a solution}\label{existence}

\begin{theorem}
  Let the assumptions \textbf{(A)}, \textbf{(B)}, and \textbf{(F)} be fulfilled and let $u_0\in V_B$, $v_0\in H$ be given. Then there exists a solution $v\in L^p(0,T;V_A)\cap \mathscr{C}([0,T];H)$ to problem~\eqref{problem} with $Kv\in \mathscr{C}([0,T];V_B)$ and $v'\in L^q(0,T;V_A^*)+L^\infty(0,T;V_B^*)$, i.e., the initial condition is fulfilled in $H$ and there exists $f\in \mathcal{F}^1(v)$ such that the equation
  \begin{equation*}
    v'+Av+BKv=f
  \end{equation*}
  holds in the sense of $L^q(0,T;V_A^*)$.
\end{theorem}

\begin{proof}
Following the proof of \cite[Theorem 3.1]{Papageorgiou91}, we want to apply the Kakutani fixed-point theorem, generalised by Glicksberg~\cite{Glicksberg} and Fan~\cite{Fan} to infinite-dimensional locally convex topological vector spaces.

First, we need a priori estimates for the solution. Assume $v\in L^p(0,T;V_A)\cap \mathscr{C}([0,T];H)$ solves problem~\eqref{problem} with the regularity stated in the theorem. Due to Lemma~\ref{MeasurabilityNemytskii}, there exists a measurable selection $f\colon [0,T]\to H$ of the mapping $t\mapsto F(t, v(t))$. The growth condition of $F$ implies
\begin{equation*}
  \begin{split}
    \Vert f(t)\Vert_H\leq a(t)+b\Vert v(t)\Vert_H^{2/q}
  \end{split}
\end{equation*}
for almost all $t\in (0,T)$, and since $a\in L^q(0,T)$ and $v\in \mathscr{C}([0,T];H)$, we have $f\in L^q(0,T;H)$. Now, test the equation
\begin{equation*}
  v' + Av +BKv = f
\end{equation*}
with $v$ and integrate the resulting equation over $(0,t)$, $t\in[0,T]$, which yields
\begin{equation} \label{testedeq}
  \int_0^t\langle v'(s)+(BKv)(s),v(s)\rangle \diff s + \int_0^t \langle Av(s),v(s)\rangle \diff s =\int_0^t \langle f(s),v(s)\rangle \diff s.
\end{equation}
Since we neither know $v'\in L^q(0,T;V_A^*)$ nor $BKv\in L^q(0,T;V_A^*)$, it is not possible to do integration by parts for each term separately. However, \cite[Lemma 4.3]{EikEmmKre} yields
\begin{equation}\label{PartInt}
  \begin{aligned}
    &\int^t_0 \langle v'(s) + (BKv)(s), v(s)\rangle \diff s = \\
    &\frac{1}{2}\|v(t)\|^2_H - \frac{1}{2}\|v_0\|^2_H + \frac{1}{2\lambda}\|(Kv)(t)\|^2_{B} -\frac{1}{2\lambda}\|u_0\|^2_{B}  - \int_0^t \langle (BKv)(s), u_0\rangle \diff s + \int_0^t \|(Kv)(s)\|^2_{B}\diff s.
  \end{aligned}
\end{equation}
%  \begin{equation} \label{testedeq}
%    \frac{1}{2}\frac{\diff}{\diff t} \vert v\vert^2 + \langle Av, v \rangle + \langle BKv, v \rangle = \langle f,v \rangle.
%  \end{equation}
%  Here, we already made use of the identity
%  \begin{equation} \label{deriv id}
%    \langle v',v\rangle=\frac{1}{2}\frac{\diff}{\diff t} \vert v\vert^2,
%  \end{equation}
%  see, e.g., \cite[Remark 7.5]{Roubicek}.
Due to the coercivity of $A$ and Young's inequality, we have
\begin{equation*}
  \begin{split}
    &\frac{1}{2}\Vert v(t)\Vert^2_H +\mu_A \int_0^t \Vert v(s)\Vert^p_{V_A} \diff s + \frac{1}{2\lambda}\|(Kv)(t)\|^2_{B}+ \int_0^t \|(Kv)(s)\|^2_{B}\diff s \\
    &\leq c_A\;\!T+ \frac{1}{2} \Vert v_0\Vert^2_H + \frac{1}{2\lambda}\Vert u_0\Vert_{B}^2 + \int_0^t \Vert f(s)\Vert_{V_A^*} \Vert v(s)\Vert_{V_A} \diff s + \int_0^t  \Vert (Kv)(s)\Vert_{B} \Vert u_0\Vert_{B} \diff s \\
    &\leq c_A\;\!T + \frac{1}{2} \Vert v_0\Vert^2_H + \frac{1}{2\lambda}\Vert u_0\Vert_{B}^2 + c\int_0^t \Vert f(s)\Vert_{V_A^*}^q\diff s + \frac{\mu_A}{2} \int_0^t \Vert v(s)\Vert^p_{V_A} \diff s \\
    &\hspace{7cm}+ \frac{1}{2} \int_0^t \Vert (Kv)(s)\Vert_{B}^2 \diff s +\frac{T}{2} \Vert u_0\Vert^2_B.
  \end{split}
\end{equation*}
After rearranging, the estimate on $F$ yields
\begin{equation*}
  \begin{split}
    &\frac{1}{2}\Vert v(t)\Vert^2_H + \frac{\mu_A}{2} \int_0^t \Vert v(s)\Vert^p_{V_A} \diff s +\frac{1}{2\lambda}\|(Kv)(t)\|^2_{B}+ \frac{1}{2} \int_0^t \Vert (Kv)(s)\Vert_{B}^2 \diff s \\
    &\leq c\left( 1+ \Vert v_0\Vert_H^2 + \Vert u_0 \Vert_B^2 + \int_0^t \Vert f(s)\Vert_{V_A^*}^q\diff s \right) \\
    &\leq c\left( 1+ \Vert v_0\Vert_H^2 + \Vert u_0 \Vert_B^2 + \int_0^t \left( a(s) + b \Vert v(s)\Vert_{H}^{2/q}\right)^q\diff s \right) \\
    &\leq c\left( 1+ \Vert v_0\Vert_H^2 + \Vert u_0 \Vert_B^2 + \Vert a\Vert_{L^q(0,T)}^q + \int_0^t\Vert v(s)\Vert_{H}^{2}\diff s \right).
  \end{split}
\end{equation*}
%... using the linearity of $B$, we get
%  \begin{equation*}
%    \frac{1}{2}\vert v(t)\vert^2 +  \int_0^t \langle A(s,v(s)),v(s)\rangle \diff s  = \frac{1}{2} \vert v_0\vert^2 + \int_0^t \langle f(s),v(s)\rangle \diff s - \int_0^t \int_0^s k(s-\tau)\;\! \langle Bv(\tau),v(\tau)\rangle\diff \tau \diff s.
%  \end{equation*}
%  The coercivity of $A(t,\cdot)$ and the Young inequality then yield
%  \begin{equation*}
%    \begin{split}
%      & \frac{1}{2}\vert v(t)\vert^2 +\mu\int_0^t \Vert v(s)\Vert^2 \diff s \\
%      &\leq \frac{1}{2} \vert v_0\vert^2 +  \int_0^t \Vert f(s)\Vert_* \;\! \Vert v(s)\Vert \diff s +  \Vert k\Vert_{L^\infty} \int_0^t \int_0^s \Vert v(\tau)\Vert_B^2 \diff \tau \diff s \\
%      &\leq \frac{1}{2} \vert v_0\vert^2 + \frac{1}{2\mu}\int_0^t \Vert f(s)\Vert^2_*\diff s+ \frac{\mu}{2} \int_0^t \Vert v(s)\Vert^2 \diff s + \Vert k\Vert_{L^\infty} \int_0^t \int_0^s \Vert v(\tau)\Vert_B^2 \diff \tau \diff s \\
%      &\leq \frac{1}{2} \vert v_0\vert^2 + \frac{1}{2\mu}\int_0^t \Vert f(s)\Vert^2_*\diff s+ \frac{\mu}{2} \int_0^t \Vert v(s)\Vert^2 \diff s + c\;\! \Vert k\Vert_{L^\infty} \int_0^t \int_0^s \Vert v(\tau)\Vert^2 \diff\tau \diff s.
%    \end{split}
%  \end{equation*}
Applying Gronwall's lemma, we obtain
\begin{equation} \label{a priori 1}
  \Vert v(t)\Vert^2_H \leq M_1
\end{equation}
for all $t\in [0,T]$, where $M_1>0$ depends on the problem data. This also immediately yields
\begin{equation} \label{a priori 2}
  \int_0^t\Vert v(s)\vert^2_{V_A} \diff s \leq M_2
\end{equation}
as well as 
\begin{equation} \label{a priori 3}
  \|(Kv)(t)\|^2_{B} \leq M_2
\end{equation}
for all $t\in [0,T]$, where $M_2>0$  also depends on the problem data.
  
We also need a priori estimates for the derivative $v'$. Due to the estimate \eqref{estimateA} and the assumptions \textbf{(F)}, we have
\begin{equation} \label{a priori 4 prep} 
  \begin{split}
    &\Vert v'\Vert_{L^q(0,T;V_A^*)+ L^\infty(0,T;V_B^*)} \\
    & \leq \max\left( \Vert Av\Vert_{L^q(0,T;V_A^*)}+ \Vert f\Vert_{L^q(0,T;V_A^*)} ,\ \Vert BKv\Vert_{L^\infty(0,T;V_B^*)} \right) \\
    & \leq \max\left( c \left(1+ \Vert v\Vert^{p-1}_{L^p(0,T;V_A)}\right) + \Vert a\Vert_{L^q(0,T)} + b\;\!\Vert v\Vert_{L^2(0,T;H)}^{2/q},\ \beta_B \;\! \Vert Kv\Vert_{L^\infty(0,T;V_B)} \right).
  \end{split}
\end{equation}
The a priori estimates \eqref{a priori 1}, \eqref{a priori 2}, and \eqref{a priori 3} above yield
\begin{equation} \label{a priori 4}
  \Vert v'\Vert_{L^q(0,T;V_A^*)+ L^\infty(0,T;V_B^*)} \leq M_3,
\end{equation}
where $M_3$ depends on the same parameters as $M_1$ and $M_2$ as well as on $\beta_B$.
  
Next, we define the truncation $\hat{F}$ of $F$ via
\begin{equation*}
  \hat{F}(t,w)=\begin{cases} F(t,w) & \text{if}\ \vert w \vert \leq M_1, \\ F(t,\frac{M_1}{\vert w \vert} w) & \text{if}\ \vert w\vert > M_1. \end{cases}
\end{equation*}
This set-valued function $\hat{F}$ has the same measurability and continuity properties as $F$: In order to prove the measurability, consider an arbitrary open subset $U\subset H$ and
\begin{equation*}
  \begin{split}
    \hat{F}^{-1}(U) &= \left\{ (t,v)\in [0,T]\times H \mid F(t, r_{M_1}(v))\cap U \neq \emptyset \right\} \\
    & = \left\{ (t,v)\in [0,T]\times H \mid F(t, v)\cap U \neq \emptyset \right\} \cap \left( [0,T] \times B_{M_1}^H \right),\\
  \end{split}
\end{equation*}
where $r_{M_1}$ is the radial retraction in $H$ to the ball $B_{M_1}^H$ of radius $M_1$. Due to the measurability of $F$, the first set is an element of $\mathcal{L}([0,T])\otimes \mathcal{B}(H)$, and since the second set is obviously an element of the same $\sigma$-algebra, $\hat{F}$ is measurable.

For proving that $\hat{F}$ fulfils the same continuity condition as $F$, let $N\subset [0,T]$ be the set of Lebesgue-measure $0$ such that the graph of $v\mapsto F(t,v)$ is sequentially closed in $H\times H_w$ for all $t\in [0,T]\setminus N$. Now, for arbitrary $t\in[0,T]\setminus N$, consider a sequence $\{(v_n,w_n)\}\subset \mathrm{graph}(\hat{F}(t,\cdot))$ with $v_n\to v$ and $w_n\wconv w$ for some $v,w\in H$. We have to show $w\in \hat{F}(t,v)$. We start with the case $\vert v\vert < M_1$. For $n$ large enough, we have $\vert v_n\vert<M_1$ and therefore $w_n\in F(t,v_n)$. The continuity property of $F$ yields $w\in F(t,v)=\hat{F}(t,v)$. Now, consider the case $\vert v\vert>M_1$. For $n$ large enough, we again have $\vert v_n\vert>M_1$. Defining $\tilde{v}_n:=\frac{M_1}{\vert v_n\vert}v_n$, we have $\tilde{v}_n\to \tilde{v}:=\frac{M_1}{\vert v\vert}v$ in $H$ and $w_n\in F(t,\tilde{v}_n)$. The continuity property of $F$ again yields $w\in F(t,\tilde{v})=\hat{F}(t,v)$. Finally, consider the case $\vert v\vert=M_1$. We know that there exists a subsequence $n'$ such that either $\vert v_{n'}\vert\leq M_1$ or $\vert v_{n'}\vert> M_1$ for all $n'$. Since we still have $w_{n'}\wconv w$, the same arguments as in the first or second case yield $w\in \hat{F}(t,v)$.

Due to the estimate on $F$ in the assumptions (\textbf{F}), we have
\begin{equation*}
  \vert \hat{F}(t,v)\vert \leq \hat{a}(t):=a(t) +b M_1^{2/q }
\end{equation*}
for almost all $t\in(0,T)$ and all $v\in H$. Now, set
\begin{equation*}
  E:= \{f\in L^q(0,T;H)\mid \vert f(t)\vert\leq \hat{a}(t) \ \text{a.e.}\}.
\end{equation*}  
We define the solution operator 
\begin{equation*}
  G\colon E\to \overline{W}(0,T):=\left\{ v\in L^p(0,T;V_A) \mid v'\in L^q(0,T;V_A^*)+L^\infty(0,T;V_B^*) \right\}
\end{equation*}
with $G(f)=v$, where $v$ is the unique solution to the problem
\begin{equation} \label{problem_f}
  \begin{split}
    v'+Av+BKv&=f, \\
    v(0)&=v_0,
  \end{split}
\end{equation}
which exists due to \cite[Theorem 4.2, Corollary 5.2]{EikEmmKre}. The aim is to show that $G$ is sequentially weakly continuous.
  
We therefore consider a sequence $\{f_n\}\subset E$ and $f\in E$ with $f_n\rightharpoonup f$ in $L^q(0,T;H)$. Following the a priori estimates \eqref{a priori 1}, \eqref{a priori 2}, \eqref{a priori 3}, and \eqref{a priori 4}, the sequence $\{v_n\}$ of corresponding solutions, i.e., $v_n=G(f_n)$, the sequence $\{v_n'\}$ of derivatives, and the sequence $\{Kv_n\}$ are bounded in $L^p(0,T;V_A)\cap L^\infty(0,T;H)$, $L^q(0,T;V_A^*)+L^\infty(0,T;V_B^*)$, and $L^\infty(0,T;V_B)$, respectively. Due to the estimate \eqref{estimateA} on $A$, this implies the boundedness of the sequences $\{Av_n\}$ in $L^q(0,T;V_A^*)$, see also \eqref{a priori 4 prep}. Since $L^p(0,T;V_A)$ is a reflexive Banach space and $L^\infty(0,T;H)$, $L^q(0,T;V_A^*)$, $L^\infty(0,T;V_B)$ as well as $L^q(0,T;V_A^*)+L^\infty(0,T;V_B^*)$ are duals of separable normed spaces, there exists a subsequence (again denoted by $n$) and $v\in L^p(0,T;V_A)\cap L^\infty(0,T;H)$, $w\in L^q(0,T;V_A^*)+L^\infty(0,T;V_B^*)$, $\tilde{a}\in L^q(0,T;V_A^*)$, and $u\in L^\infty(0,T;V_B)$ such that
\begin{equation*}
  \begin{split}
    v_n &\wconv v \phantom{vau} \text{in}\ L^p(0,T;V_A), \\
    v_n &\wsconv v \phantom{vau} \text{in}\ L^\infty(0,T;H), \\
    v_n' &\wsconv \hat{v} \phantom{vau} \text{in}\ L^q(0,T;V_A^*)+L^\infty(0,T;V_B^*), \\
    Av_n &\wsconv \tilde{a} \phantom{vvu} \text{in}\ L^q(0,T;V_A^*), \\
    Kv_n &\wsconv u \phantom{vva} \text{in}\ L^\infty(0,T;V_B),
  \end{split}
\end{equation*}
as $n\to \infty$. We obviously have $\hat{v}=v'$.

In order to show that $G$ is sequentially weakly continuous, we have to pass to the limit in the equation
\begin{equation} \label{equation_n}
  v_n' + Av_n +BKv_n = f_n
\end{equation}
and show that $v$ is a solution to problem \eqref{problem_f}. First, we want to show $u=Kv$. We know that the operator $\hat{K}\colon L^2(0,T;H)\to L^2(0,T;H)$ with $\hat{K}w:=Kw-u_0$ is well-defined, linear, and bounded, see Lemma \ref{estimateK}. Thus it is weakly-weakly continuous and $v_n\wsconv v$ in $L^\infty(0,T;H)$ (and therefore $v_n\wconv v$ in $L^2(0,T;H)$) implies $Kv_n-Kv=\hat{K}v_n-\hat{K}v\wconv 0$ in $L^2(0,T;H)$. This yields $u=Kv$. Due to the linearity and boundedness of $B\colon V_B\to V_B^*$ and thus its weakly*-weakly* continuity, we also have $BKv_n\wsconv Bu=BKv$ in $L^\infty(0,T;V_B^*)$.

Next, let us show $v(0)=v_0$ and $v_n(T)\wconv v(T)$ in $H$. Due to estimate \eqref{a priori 1}, the sequence $\{v_n(T)\}$ is bounded in $H$, so there exists $v_T\in H$ such that, up to a subsequence, $v_n(T)\wconv v_T$ in $H$. Now, consider $\phi\in \mathscr{C}^1([0,T])$, $w\in V$ (recall that $V=V_A\cap V_B$). Since $v_n$ solves \eqref{equation_n} in the weak sense and $v$ solves
\begin{equation} \label{equation_a}
  v'+\tilde{a}+BKv=f
\end{equation}
in the weak sense, we have
\begin{equation*}
  \begin{split}
    \left(v_n(T),w\right) \phi(T) - \left(v_n(0),w\right)\phi(0) &= \int_0^T \langle v_n'(t), w\rangle\;\! \phi(t)\diff t + \int_0^T \langle v_n(t),w\rangle\;\! \phi'(t)\diff t \\
    & = \int_0^T \langle f_n - Av_n - BKv_n, w \rangle \;\! \phi(t) \diff t + \int_0^T \langle v_n(t),w\rangle\;\! \phi'(t)\diff t \\
    &\rightarrow \int_0^T \langle f - \tilde{a} - BKv, w \rangle \;\! \phi(t) \diff t + \int_0^T \langle v(t),w\rangle\;\! \phi'(t)\diff t \\
    &= \int_0^T \langle v', w \rangle \;\! \phi(t) \diff t + \int_0^T \langle v(t),w\rangle\;\! \phi'(t)\diff t \\
    &= \left(v(T),w\right) \phi(T) - \left(v(0),w\right)\phi(0) \phantom{\int_0^T}
  \end{split}
\end{equation*}
as $n\to \infty$. Choosing $\phi(t)=1-\frac{t}{T}$, this yields $(v_n(0),w)\to (v(0),w)$ for all $w\in V$. Due to $v_n(0)=v_0$ for all $n\in \mathbb{N}$, we have $v(0)=v_0$. Choosing $\phi(t)=\frac{t}{T}$, we get $(v_n(T),w)\to (v(T),w)$ for all $w\in V$ and therefore also $v_T=v(T)$ in $H$.

Next, let us show $(Kv_n)(T)\wconv (Kv)(T)$ in $V_B$. Estimate \eqref{a priori 3} implies the boundedness of the sequence $\{(Kv_n)(T)\}$ in $V_B$, therefore there exists $u_T\in V_B$ such that, up to a subsequence, $(Kv_n)(T)\wconv u_T$ in $V_B$. Consider again $\phi(t)=\frac{t}{T}$, $w\in V$. Due to relation \eqref{RelationK}, we have
\begin{equation*}
  \begin{split}
    ((Kv_n)(T), w)&= \int_0^T \langle (Kv_n)'(t), w\rangle \frac{t}{T} \diff t + \int_0^T \langle (Kv_n)(t),w\rangle \frac{1}{T} \diff t \\
    &=\lambda \int_0^T \langle v_n(t)-((Kv_n)(t)-u_0),w\rangle \frac{t}{T} \diff t + \int_0^T \langle (Kv_n)(t),w\rangle \frac{1}{T} \diff t \\
    &\rightarrow \lambda \int_0^T \langle v(t)-((Kv)(t)-u_0),w\rangle \frac{t}{T} \diff t + \int_0^T \langle (Kv)(t),w\rangle \frac{1}{T} \diff t \\
    &= \int_0^T \langle (Kv)'(t), w\rangle \frac{t}{T} \diff t + \int_0^T \langle (Kv)(t),w\rangle \frac{1}{T} \diff t \\
    &= ((Kv)(T), w) \phantom{\int_0^T}
  \end{split}
\end{equation*}
as $n\to \infty$. This immediately yields $u_T=(Kv)(T)$.

In order to show that $v$ is a solution to problem \eqref{problem_f}, it remains to show $\tilde{a}=Av$. Using the integration-by-parts formula \cite[Lemma 4.3]{EikEmmKre}, we obtain
\begin{equation*}
  \begin{split}
    \langle Av_n,v_n\rangle &= \langle f_n, v_n\rangle - \langle v_n'+BKv_n, v_n\rangle \\
    &= \langle f_n, v_n\rangle - \frac{1}{2}\|v_n(T)\|^2_H + \frac{1}{2}\|v_0\|^2_H - \frac{1}{2\lambda}\|(Kv_n)(T)\|^2_{B} +\frac{1}{2\lambda}\|u_0\|^2_{B} \\
    &\hspace{5cm} + \int_0^T \langle (BKv_n)(s), u_0\rangle \diff s - \Vert Kv_n\Vert^2_{L^2(0,T;B)}.
  \end{split}
\end{equation*}
Since we have $v_n\wconv v$ in $\bar{W}(0,T)$ and since the embedding $\bar{W}(0,T)\subset L^p(0,T;H)$ is compact (see, e.g., Roub\'i\v{c}ek~\cite[Lemma~7.7]{Roubicek}), there exists a subsequence, again denoted by $n$, such that $v_n\to v$ in $L^p(0,T;H)$. This yields $\langle f_n, v_n\rangle\to \langle f,v\rangle$.\footnote{Here, we need the (in comparison to the single-valued case stronger) assumptions that the embedding $V_A\subset H$ is compact and that $f(t)\in H$ in order to identify the limit of the sequence $\{\langle f_n, v_n\rangle\}$.} We also obviously have 
\begin{equation*}
  \int_0^T \langle (BKv_n)(s), u_0\rangle \diff s\to \int_0^T \langle (BKv)(s), u_0\rangle \diff s.
\end{equation*}
Due to the convergences $v_n(T)\wconv v(T)$ in $H$, $(Kv_n)(T)\wconv (Kv)(T)$ in $V_B$ as well as $Kv_n\wsconv Kv$ in $L^\infty(0,T;V_B)$ (and thus $Kv_n\wconv Kv$ in $L^2(0,T;V_B)$) and the lower semicontinuity of the norm, we obtain
\begin{equation*}
  \begin{split}
    \limsup_{n\to\infty}\;\! \langle Av_n,v_n\rangle &\leq  \langle f, v\rangle - \frac{1}{2}\|v(T)\|^2_H + \frac{1}{2}\|v_0\|^2_H - \frac{1}{2\lambda}\|(Kv)(T)\|^2_{B} +\frac{1}{2\lambda}\|u_0\|^2_{B} \\
    &\hspace{5cm} + \int_0^T \langle (BKv)(s), u_0\rangle \diff s - \Vert Kv\Vert^2_{L^2(0,T;B)} \\
    &= \langle f,v\rangle - \langle v'+BKv,v\rangle,
  \end{split}
\end{equation*}
using again the integration-by-parts formula \cite[Lemma 4.3]{EikEmmKre}. As $v$ solves \eqref{equation_a} in the weak sense, we have
\begin{equation} \label{limsup A}
  \limsup_{n\to\infty}\;\! \langle Av_n,v_n\rangle \leq \langle \tilde{a},v\rangle.
\end{equation}
Now, for arbitrary $w\in L^p(0,T;V_A)$, the monotonicity of $A$ implies
\begin{equation*}
  \begin{split}
    \langle Av_n, v_n\rangle &= \langle Av_n-Aw,v_n-w\rangle + \langle Aw, v_n-w\rangle +\langle Av_n, w\rangle \\
    &\geq \langle Aw, v_n-w \rangle + \langle Av_n, w\rangle.
  \end{split}
\end{equation*}
Therefore, we obtain
\begin{equation*}
  \liminf_{n\to\infty} \langle Av_n, v_n\rangle \geq \langle Aw, v-w\rangle + \langle \tilde{a},w\rangle
\end{equation*}
and, together with \eqref{limsup A},
\begin{equation*}
  \langle Aw-\tilde{a}, v-w\rangle\leq 0.
\end{equation*}
Choosing $w=v\pm rz$ for an arbitrary $z\in L^p(0,T;V_A)$ and $r>0$ and using the hemicontinuity as well as the growth condition of $A\colon V_A\to V_A^*$, Lebesgue's theorem on dominated convergence yields for $r\to 0$
\begin{equation*}
  \langle Av -\tilde{a} , z\rangle =0
\end{equation*}
for all $z\in L^p(0,T;V_A)$, which implies $\tilde{a}=Av$.

%This is done analogously to the proof of \cite[Theorem 4.2]{EikEmmKre} with the slight difference that we have $f_n\wconv f$ in $L^q(0,T;H)$ instead of $f_n\to f$ in $L^q(0,T;V_A^*)$. However, considering that we have the compact embedding $\bar{W}(0,T)\subset L^p(0,T;H)$ (see, e.g., Roub\'i\v cek, we can use the strong convergence of $v_n$ in $L^p(0,T;H)$ (up to a subsequence) to pass to the limit in the term $\langle f_n, v_n\rangle$, where $\langle\cdot,\cdot\rangle$ denotes the duality pairing between $L^q(0,T;H)$ and $L^p(0,T;H)$.

As the last step of this proof, consider the operator $R\colon E \to P_{fc}(E)$ with $R(f)=\mathcal{F}^1(G(f))$, where the set $\mathcal{F}^1(G(f))$ is meant with respect to the truncation $\hat{F}$ instead of $F$, i.e.,
\begin{equation*}
  \mathcal{F}^1(G(f))=\left\{ f\in L^1(0,T;H) \mid f(t)\in \hat{F}(t,(G(f))(t))\ \text{a.e. in}\ (0,T) \right\}\!.
\end{equation*}
Following the proof of \cite[Theorem 3.1]{Papageorgiou91}, this operator is upper semicontinuous on $E$ equipped with the weak topology. Thus, we can apply the generalisation of the Kakutani fixed-point theorem (see Glicksberg~\cite{Glicksberg} and Fan~\cite{Fan}) to obtain the existence of $f\in E$ such that $f\in R(f)=\mathcal{F}^1(G(f))$. This implies that $v=G(f)$ solves the problem \eqref{problem} with the right-hand side $\hat{F}$. However, due to the a priori estimate \eqref{a priori 1}, we have $\hat{F}(t,v(t))=F(t,v(t))$ for all $t\in[0,T]$ which proves the assertion.\qed

%Next, we want to prove that $u_n(t)\to u(t)$ in $H$ for all $t\in[0,T]$. Therefore, let $\varphi\in \mathscr{C}^1([0,T])$, $v\in V$ and $t\in[0,T]$. We have
%\begin{equation*}
%  \begin{split}
%    (u_n(t),v\;\!\varphi(t))-(u_n(0), v\;\!\varphi(0)) &= \int_0^t \left( \langle u_n'(s), v\;\!\varphi(s)\rangle+ \langle u_n(s), v\;\! \varphi'(s)\rangle\right)\diff t \\
%    & =\int_0^t \left( \langle f_n(s)-Au_n(s) -(BKu_n)(s), v\;\!\varphi(s)\rangle + \langle u_n(s), v\;\! \varphi'(s)\rangle\right)\diff t. \\
%  \end{split}
%\end{equation*}
%The latter integral converges to 
%\begin{equation*}
%  t
%\end{equation*}
%
%Now we test the equation
%\begin{equation*}
%  u_n' + Au_n +BKu_n = f_n
%\end{equation*}
%with $u_n -u$ and integrate over $(0,t)$, $t\in[0,T]$, which yields
%\begin{equation} \label{testedeq n}
%    \frac{1}{2} \vert u_n(t)-u(t)\vert^2 + \int_0^t\langle u'(s), u_n(s)-u(s)\rangle \diff s + \langle Au_n, u_n-u \rangle + \langle BKu_n, u_n-u \rangle = \langle f_n,u_n-u \rangle.
%\end{equation}
  
\end{proof}

%%%%%%%%%%%%%%%%%%%%%%%%%%%%%%%%%%%%%%%%%%%%%%%%%%%%%%%%

%%%%%%%%%%%%%%%%%%%%%%%%%%%%%%%%%%%%%%%%%%%%%%%%%%%%%%%%
%%%%%%%%%%%%%%%%%%%%%%%%%%%%%%%%%%%%%%%%%%%%%%%%%%%%%%%%


\begin{thebibliography}{10}

\bibitem{AmannEscher}
H.~Amann and J.~Escher.
\newblock {\em Analysis. {III}}.
\newblock Birkh\"{a}user, Basel, 2001.

\bibitem{AubinCellina}
J.-P. Aubin and A.~Cellina.
\newblock {\em Differential inclusions}.
\newblock Springer, Berlin, 1984.

\bibitem{AubinFrankowska}
J.-P. Aubin and H.~Frankowska.
\newblock {\em Set-{V}alued {A}nalysis}.
\newblock Birkh\"{a}user, Boston, MA, 1990.

\bibitem{BeynEmmRie}
W.-J. Beyn, E.~Emmrich, and J.~Rieger.
\newblock Semilinear parabolic differential inclusions with one-sided
  {L}ipschitz nonlinearities.
\newblock {\em J. Evol. Equ.}, 18(3):1319--1339, 2018.

\bibitem{Brezis}
H.~Br{\'e}zis.
\newblock {\em Analyse Fonctionnelle: Th{\'e}orie et Applications}.
\newblock Dunod, Paris, 1999.

\bibitem{CastaingValadier}
C.~Castaing and M.~Valadier.
\newblock {\em Convex {A}nalysis and {M}easurable {M}ultifunctions}.
\newblock Lecture Notes in Mathematics, Vol. 580. Springer, Berlin, 1977.

\bibitem{CrandallLondenNohel}
M.~G. {Crandall}, S.-O. {Londen}, and J.~A. {Nohel}.
\newblock An abstract nonlinear {V}olterra integrodifferential equation.
\newblock {\em J. Math. Anal. Appl.}, 64:701--735, 1978.

\bibitem{Cushing}
J.~M. {Cushing}.
\newblock {\em Integrodifferential Equations and Delay Models in Population
  Dynamics}.
\newblock Lecture Notes in Biomathematics, Vol. 20. Springer, Berlin, 1977.

\bibitem{Deimling}
K.~Deimling.
\newblock {\em Multivalued differential equations}.
\newblock de Gruyter, Berlin, 1992.

\bibitem{DenMigPapa}
Z.~Denkowski, S.~Mig\'{o}rski, and N.~S. Papageorgiou.
\newblock {\em An {I}ntroduction to {N}onlinear {A}nalysis: {T}heory}.
\newblock Kluwer Academic Publishers, Boston, MA, 2003.

\bibitem{Desch88}
W.~Desch, R.~Grimmer, and W.~Schappacher.
\newblock Wellposedness and wave propagation for a class of integrodifferential
  equations in {B}anach space.
\newblock {\em J. Differ. Equations}, 74(2):391--411, 1988.

\bibitem{DiestelUhl}
J.~Diestel and J.~J. Uhl.
\newblock {\em Vector Measures}.
\newblock American Mathematical Society, Providence, Rhode Island, 1977.

\bibitem{Edwards96}
D.~A. Edwards.
\newblock Non-{F}ickian diffusion in thin polymer films.
\newblock {\em J. Polym. Sci., Part B: Polym. Phys.}, 34:981--997, 1996.

\bibitem{EdwardsCohen95}
D.~A. Edwards and D.~S. Cohen.
\newblock A mathematical model for a dissolving polymer.
\newblock {\em AIChE Journal}, 41(11):2345--2355, 1995.

\bibitem{EikEmmKre}
A.~Eikmeier, E.~Emmrich, and H.-C. Kreusler.
\newblock Nonlinear evolution equations with exponentially decaying memory:
  {E}xistence via time discretisation, uniqueness, and stability.
\newblock {\em Comput. Methods Appl. Math.}, 2018.
\newblock Published online.

\bibitem{EmmrichThalhammer1}
E.~Emmrich and M.~Thalhammer.
\newblock Doubly nonlinear evolution equations of second order: Existence and
  fully discrete approximation.
\newblock {\em J. Differ. Equations}, 251:82--118, 2011.

\bibitem{Fan}
K.~Fan.
\newblock Fixed-point and minimax theorems in locally convex topological linear
  spaces.
\newblock {\em Proc. Nat. Acad. Sci. U. S. A.}, 38:121--126, 1952.

\bibitem{FedotovIomin08}
S.~Fedotov and A.~Iomin.
\newblock Probabilistic approach to a proliferation and migration dichotomy in
  tumor cell invasion.
\newblock {\em Phys. Rev. E}, 77(3):031911, 2008.

\bibitem{GGZ}
H.~Gajewski, K.~Gr\"oger, and K.~Zacharias.
\newblock {\em Nichtlineare Operatorgleichungen und
  Ope\-rator\-differentialgleichungen}.
\newblock Akademie-Verlag, Berlin, 1974.

\bibitem{Glicksberg}
I.~L. Glicksberg.
\newblock A further generalization of the {K}akutani fixed theorem, with
  application to {N}ash equilibrium points.
\newblock {\em Proc. Amer. Math. Soc.}, 3:170--174, 1952.

\bibitem{Gori_etal}
C.~Gori, V.~Obukhovskii, P.~Rubbioni, and V.~Zvyagin.
\newblock Optimization of the motion of a visco-elastic fluid via multivalued
  topological degree method.
\newblock {\em Dynam. Systems Appl.}, 16(1):89--104, 2007.

\bibitem{MacCamy77a}
R.~C. MacCamy.
\newblock An integro-differential equation with application to heat flow.
\newblock {\em Quart. Appl. Math.}, 35:1--19, 1977.

\bibitem{MacCamy77b}
R.~C. MacCamy.
\newblock A model for one-dimensional nonlinear viscoelasticity.
\newblock {\em Quart. Appl. Math.}, 35:21--33, 1977.

\bibitem{MehrabianAbousleiman11}
A.~Mehrabian and Y.~Abousleiman.
\newblock General solutions to poroviscoelastic model of hydrocephalic human
  brain tissue.
\newblock {\em J. Theor. Biol.}, 291:105--118, 2011.

\bibitem{Miller78}
R.~K. Miller.
\newblock An integrodifferential equation for rigid heat conductors with
  memory.
\newblock {\em J. Math. Anal. Appl.}, 66:313--332, 1978.

\bibitem{ObZecZvy}
V.~Obukhovski\u{\i}, P.~Zecca, and V.~Zvyagin.
\newblock Optimal feedback control in the problem of the motion of a
  viscoelastic fluid.
\newblock {\em Topol. Methods Nonlinear Anal.}, 23(2):323--337, 2004.

\bibitem{ORegan}
D.~O'Regan.
\newblock Multivalued differential equations in {B}anach spaces.
\newblock {\em Comput. Math. Appl.}, 38(5-6):109--116, 1999.

\bibitem{Papageorgiou88b}
N.~S. Papageorgiou.
\newblock Existence and convergence results for integral inclusions in {B}anach
  spaces.
\newblock {\em J. Integral Equations Appl.}, 1(2):265--285, 1988.

\bibitem{Papageorgiou88a}
N.~S. Papageorgiou.
\newblock Volterra integral inclusions in {B}anach spaces.
\newblock {\em J. Integral Equations Appl.}, 1(1):65--81, 1988.

\bibitem{Papageorgiou91}
N.~S. Papageorgiou.
\newblock Nonlinear {V}olterra integrodifferential evolution inclusions and
  optimal control.
\newblock {\em Kodai Math. J.}, 14(2):254--280, 1991.

\bibitem{Papageorgiou91b}
N.~S. Papageorgiou.
\newblock Volterra integrodifferential inclusions in reflexive {B}anach spaces.
\newblock {\em Funkcial. Ekvac.}, 34(2):257--277, 1991.

\bibitem{PapaWin}
N.~S. Papageorgiou and P.~Winkert.
\newblock {\em Applied nonlinear functional analysis}.
\newblock De Gruyter, Berlin, 2018.

\bibitem{Roubicek}
T.~Roub\'{\i}\v{c}ek.
\newblock {\em Nonlinear Partial Differential Equations with Applications}.
\newblock Birkh\"auser, Basel, 2005.

\bibitem{ShawWhiteman98}
S.~Shaw and J.~R. Whiteman.
\newblock Some partial differential {V}olterra equation problems arising in
  viscoelasticity.
\newblock In {\em Proceedings of Equadiff 9}, pp. 183--200. Masaryk
  University, Brno, 1998.

\bibitem{ZeidlerIIB}
E.~Zeidler.
\newblock {\em Nonlinear Functional Analysis and its Applications, II/B:
  Nonlinear Monotone Operators}.
\newblock Springer, New York, 1990.

\end{thebibliography}
\end{document}